\documentclass[twocolumn]{autart} 

\usepackage{cite}
\usepackage{amsmath,amssymb,amsfonts}
\usepackage{algorithmic}
\usepackage{textcomp}
\usepackage{color}
\usepackage{bm}
\usepackage{graphicx}
\usepackage{enumerate}
\usepackage{multicol}
\usepackage{booktabs,tabularx,makecell,array}

\newcolumntype{L}{>{\raggedright\arraybackslash}X}

\newtheorem{lemma}{Lemma}
\newtheorem{assumption}{Assumption}

\newtheorem{remark}{Remark}

\newtheorem{theorem}{Theorem}

\begin{document}

\begin{frontmatter}
%\runtitle{Insert a suggested running title}  % Running title for regular 
% papers but only if the title  
% is over 5 words. Running title 
% is not shown in output.

\title{Primal-dual algorithm for distributed optimization: A dissipativity-based perspective\thanksref{footnoteinfo}} % Title, preferably not more 
% than 10 words.

\thanks[footnoteinfo]{This paper was not presented at any IFAC 
meeting. \\
$*$ Corresponding author. %Tel. +XXXIX-VI. 
%Fax +XXXIX-VI.
}

\author{Weijian Li},
\ead{wli26@nd.edu}
\author{Panos J. Antsaklis},
\ead{pantsak1@nd.edu}
\author{Hai Lin$^*$}
\ead{hlin1@nd.edu}

%\author{Weijian Li, Panos J. Antsaklis, Hai Lin}\ead{wli26@nd.edu}(W. Li)
%\ead{pantsakl@nd.edu}
%\ead{hlin1@nd.edu}   % Add the e-mail address 

\address{Department of Electrical Engineering, University of Notre Dame, Notre Dame, IN 46556, USA}  % Please supply                                                     % full addresses
%\address[Baiae]{The White House, Baiae}        % here.

                             % keyword list or with the 
% help of the Automatica 
% keyword wizard

%%%%%%
\begin{abstract}
We study a continuous-time primal-dual algorithm for distributed optimization with nonconvex local cost functions over weight-unbalanced digraphs, and analyze its performance from a dissipativity-based perspective.
We first reformulate the algorithm as a Lur'e type system, consisting of a linear subsystem that relies on the communication topology and the algorithm gains, and a static nonlinear gradient feedback. We then show that the linear subsystem is dissipative with respect to a suitable supply rate, while the nonlinear feedback is not passive. Finally, we establish that, by properly selecting the gains or appropriately designing the communication network, this algorithm converges to an equilibrium at an exponential rate, and thus, achieves an optimal solution to the distributed problem.
This work provides new insights into the roles of the network topology, algorithm gains, and cost functions in the performance of a distributed algorithm, and complements existing results from a different viewpoint.
\end{abstract}

\end{frontmatter}

%%%%
%%%%-----------------
\section{Introduction}

Recent years have witnessed a flurry of research interest in
distributed optimization, motivated by its applications in diverse areas, including distributed estimation, sensor networks, power systems, and machine learning \cite{nedic2018distributed, yang2019survey,yuan2024multi}. The basic idea is that, in a multi-agent system, all agents cooperate to minimize a sum of local costs with their local information and neighbors’ states. A variety of distributed algorithms have been proposed to ensure all agents reach consensus, and meanwhile, achieve optimality. To mention just a few, there are distributed subgradient algorithm \cite{nedic2009distributed},  alternating
direction method of multipliers (ADMM) \cite{neal2011distributed}, distributed dual averaging algorithm \cite{duchi2011dual}, and zeroth-order methods \cite{sahu2020decentralized}.
Among them, first-order primal-dual methods have been widely employed, as they are powerful for constrained optimization problems, and easy to implement in distributed manners \cite{arrow1958studies, wang2010control, gharesifard2013distributed}.
For instance, the convergence of a distributed primal-dual algorithm was studied in \cite{gharesifard2013distributed} over weight-balanced graphs.
Subsequently, the primal-dual framework was extended to constrained optimization with nonsmooth cost functions
in \cite{zeng2016distributed}, and its discrete-time counterpart was explored in \cite{lei2016primal}.
In \cite{yi2020distributed}, the algorithm was further generalized to solve distributed online convex optimization with coupled constraints.
Exponential convergence of primal–dual methods was established in \cite{liang2019exponential} without strong convexity. In \cite{li2017distributed, li2020projection}, distributed adaptive primal-dual algorithms were proposed for a convex sum of nonconvex cost functions.
The convergence of primal–dual algorithms is well understood under convex local cost functions, but remains unclear in the absence of convexity.

Indeed, distributed algorithms, including distributed primal-dual methods, can be viewed as dynamical systems. 
Hence, Lyapunov approaches have been broadly employed to establish their convergence, thanks to the well-developed stability theory \cite{zeng2016distributed, gharesifard2013distributed, kia2015distributed}.
However, they often relies on complicated Lyapunov function candidates, which are closely tied to the underlying problem and algorithm structures \cite{khalil2002nonlinear}, and moreover, the derived results may be conservative.
More recently, contraction theory has been introduced as an alternative tool to analyze optimization algorithms.
In particular, strong contractivity has been established for distributed convex optimization algorithms, providing global exponential convergence guarantees and robustness to disturbances \cite{davydov2025time, gokhale2023contractivity}.
However, both Lyapunov-based approaches and contraction theory typically treat an algorithm as an integrated dynamical system, and provide limited insight into the interaction among its individual components.
The performance of a distributed algorithm is jointly influenced by the communication topology, cost functions and algorithm parameters. This interdependence highlights the importance of developing a systematic framework that captures the distributed interconnection structure, and clarifies the roles of these components in shaping the properties of the algorithm.

Dissipativity is a powerful tool to address the stability of interconnected dynamical systems, as well as to design controllers with performance guarantees in both centralized and decentralized scenarios \cite{sepulchre2012constructive, khalil2002nonlinear,zakeri2022passivity, welikala2025decentralized}. 
By decomposing optimization algorithms into interconnected dynamical systems, dissipativity-based approaches have been explored to analyze their performance \cite{lessard2022analysis}.
This viewpoint provides a deeper understanding of existing algorithms by explaining why they work under certain settings, and further guides the development of novel methods.
For instance, a unifying passivity-based framework was provided in \cite{wen2004unifying} for network flow problems, and a generalization of the primal-dual algorithm was proposed in \cite{yamashita2020passivity} for non-strictly convex cost functions.
In \cite{hatanaka2018passivity}, a passivity-based distributed algorithm was designed to handle communication delays, and subsequently, an input-feedforward-passivity-based distributed algorithm was explored over jointly connected
balanced digraphs in \cite{li2020input}.
In \cite{li2025passivity}, a class of generalized gradient-play dynamics was developed based on passivity.
However, most existing works design distributed algorithms under convex local cost functions and given communication topologies. 

Inspired by the above observations, we study a primal-dual algorithm for distributed optimization with nonconvex local cost functions from a dissipativity-based perspective.
Our contributions are at least twofold.

\begin{itemize}
\item Unlike the primal–dual algorithms studied in \cite{wang2010control,gharesifard2013distributed, zeng2016distributed, lei2016primal},
we only suppose that the global cost function is convex, while each local cost may not be \cite{pang2022gradient, li2017distributed, li2020projection}. Moreover, the communication topology is allowed to be weight-unbalanced \cite{li2020cooperative, cheng2022distributed}.
We investigate the performance of this algorithm from a novel 
dissipativity-based perspective. In particular, we characterize how the communication network, the algorithm gains, and the local costs jointly influence its convergence.

\item We decompose the algorithm into a Lur'e system, in which the linear subsystem depends on the network structures and the algorithm gains, and the nonlinear feedback arises from gradients of the local costs.
Such a decomposition differs from those in \cite{li2020input,hatanaka2018passivity, wang2010control}.
We find that the linear subsystem is dissipative with a supply rate, whereas the nonlinear subsystem is not passive due to the nonconvexity of the local costs.
This is an alternative result to \cite{gharesifard2013distributed, zeng2016distributed, wang2010control, kia2015distributed}, in that the dissipativity properties are explored and exploited.
Furthermore, for given local costs,  we show that the algorithm solves the distributed optimization problem by properly selecting algorithm gains or by appropriately designing the communication graphs. Thus, the applicability of primal-dual methods is significantly extended.
\end{itemize}

This paper is organized as follows. Section 2 introduces the necessary preliminaries. 
Section 3 formulates the problem and presents the distributed primal-dual algorithm.
The main results are provided in Section 4.
Numerical simulations are carried out in Section 5.
Finally, concluding remarks are given in Section 6.

%%%%
%%%%-----------------
\section{Preliminaries}

In this section, we introduce the notation and review the necessary concepts from convex optimization, graph theory, and dissipativity.

\emph{Notation}: Let $\mathbb R$ be the set of real numbers, and $\mathbb R^m$ be the set of $m$-dimensional real column vectors. Let $\bm{1}_m (\bm{0}_m)$ be the $m$-dimensional vector with all entries of 1 (0), $\bm 0_{m \times n}$ be the $m$-by-$n$ zero matrix, and $I_n$ be the $n$-by-$n$ identity matrix. We simply write $\bm{0}$ for zero vectors/matrices of appropriate dimensions when there is no confusion. 
Denote by $\mathbb S^n_{++}$ ($\mathbb S^n_+$) the set of $n$-by-$n$ positive (semi-) definite matrices.  
We use $X \!\succ \!(\succeq) \bm 0$ to denote $X \in \mathbb S^n_{++} (\mathbb S^n_+)$.
Let $(\cdot)^\top$, $\otimes$, and $\Vert \cdot \Vert$ be the transpose, the Kronecker product, and the Euclidean norm. The Euclidean inner product of $x$ and $y$ is
$x^\top y$ or $\langle x, y \rangle$.
Denote by ${\rm diag}\{x\}$ the diagonal matrix whose diagonal elements are the entries of the vector $x$.

\emph{Convex Analysis}: Let $f: \mathbb{R}^m \to \mathbb{R}$ be a differentiable function, and $\nabla f(x)$ denote its gradient at $x$.
Then $f$ is said to be convex if $f(y) \ge f(x) + \langle y - x, \nabla f(x)\rangle,~\forall x, y \in \mathbb R^m$, and $\mu$-strongly convex if there exists $\mu>0$ such that $\langle y - x, \nabla f(y) - \nabla f(x)\rangle \ge \mu \Vert x - y\Vert^2, ~\forall x, y \in \mathbb R^m$.
The mapping $\nabla f(x)$ is $l$-Lipschitz continuous if
$\Vert \nabla f(x) - \nabla f(y)\Vert \le l \Vert x - y \Vert, \forall x, y \in \mathbb R^m$.

\emph{Graph Theory}: Consider a multi-agent network modeled by a weighted directed graph (digraph) $\mathcal G(\mathcal V, \mathcal E)$, where $\mathcal V = \{1, \dots, N\}$ is the node set, $\mathcal E \subset \mathcal V \times \mathcal V$ is the edge set. 
An edge $(j, i) \in \mathcal E$ indicates that node $i$ can receive the information from node $j$, but not vice versa.
The adjacency matrix $\mathcal A = [a_{ij}] \in \mathbb R^{N \times N}$, associated with $\mathcal G$, is defined as $a_{ij} > 0$ if $(j, i) \in \mathcal E$, and $a_{ij} = 0$ otherwise.
Suppose that there are no self-loops in $\mathcal G$, i.e., $a_{ii} = 0$.
A path is a sequence of nodes connected by edges.
A digraph is strongly connected if there is
a path between any pair of distinct nodes.
The Laplacian matrix $L_G$ is $L_G = \mathcal D - \mathcal A$, where $\mathcal D = {\rm diag}\{d_1, \dots, d_N\} \in \mathbb R^{N \times N}$, and $d_i = \sum_{j \in \mathcal V} a_{ij}$.
A digraph $\mathcal G$ is weight-balanced iff
$\bm{1}_N^\top L_G = \bm{0}_N^\top$.
%, and $\mathcal G$ is undirected if $(j, i) \in \mathcal E$ anytime $(i, j) \in \mathcal E$.
The following lemma is derived from \cite{li2017cooperative}.
\begin{lemma}
\label{lem:graph}
Let $\mathcal G$ be a strongly connected digraph with Laplacian matrix $L_G$. The following statements hold.
\begin{enumerate}[i)]
\item $L_G$ has a simple zero eigenvalue corresponding to the right eigenvector $\bm{1}_N$, and all its nonzero eigenvalues have positive real part.

\item There is a positive left eigenvector $r \in \mathbb R^N$ associated with the zero eigenvalue of $L_G$ such that $r^\top L_G = \bm{0}_N^\top$ and $\bm 1_N^\top r= 1$.

\item Define $\tilde L = RL_G + L_G^\top R$, where $R = {\rm diag}\{r\}$. Let $\rho(L_G)$ be the generalized algebraic connectivity of $\mathcal G$ as
$\rho(L_G) \!=\! \min_{r^\top x = 0, x \not= \bm{0}} \frac {x^\top \tilde L x}{2 x^\top R x}$.
Then $\tilde L$ is positive semidefinite, and $\rho(L_G) > 0$.
\end{enumerate}
\end{lemma}

\emph{Dissipativity}:
Consider a dynamical system 
\begin{equation}
\label{dyn:nonlinear}
\dot x = f(x, u), ~y = h(x)
\end{equation}
where $x \in \mathbb R^n$, $u\in \mathbb R^p$, $y \in \mathbb R^m$, $f: \mathbb R^n \times \mathbb R^p \rightarrow \mathbb R^n$ and $h: \mathbb R^n \rightarrow \mathbb R^m$ are continuously differentiable,
$f(\bm 0, \bm 0) = \bm 0$ and $h(\bm 0) = \bm 0$.
The system is said to be dissipative with respect to a supply rate $s: \mathbb R^p \times \mathbb R^m \rightarrow \mathbb R$ if there exists a continuously differentiable storage function $V: \mathbb R^n \rightarrow \mathbb R$ such that $V(\bm 0) = 0$, $V(x) \ge 0$ and $\dot V = \nabla_x V \cdot f(x, u) \le s(u, y)$ for all $(x, u) \in \mathbb R^n \times \mathbb R^m$.
If $s(u, y) = u^\top y$, then (\ref{dyn:nonlinear}) is passive. See \cite{sepulchre2012constructive,khalil2002nonlinear,  arcak2022compositional} for further details on dissipativity.

%%%%
%%%%----------------- 
\section{Formulation and algorithm}

In this section, we formulate the distributed optimization problem, and introduce a primal-dual algorithm.

%%%%%
%%%%%----------------- 
%\subsection{Problem statement}

Consider a network of $N$ agents interacting over a digraph $\mathcal G(\mathcal V, \mathcal E)$, where $\mathcal V = \{1, \dots, N\}$ and $\mathcal E \subset \mathcal V \times \mathcal V$.
Each agent $i \in \mathcal V$ is associated with a
local cost function $f_i: \mathbb R^m \rightarrow \mathbb R$. All the agents cooperate to reach a consensus solution that minimizes the global cost function 
$ \tilde f(x) := \sum_{i \in \mathcal V} f_i(x)$, where $x \in \mathbb R^m$ is the decision variable. 
To be strict, the problem is formulated as
\begin{equation}
\label{form}
\min~ \tilde f(x) := \sum\nolimits_{i \in \mathcal V} f_i(x).
\end{equation}
We make the following standing assumptions.
\begin{assumption}
\label{ass:convex}
Each local cost function $f_i$ is differentiable with $l_i$-Lipschitz continuous gradient on $\mathbb R^m$,
and the global cost function $\tilde f$ is differentiable and $\mu$-strongly convex over $\mathbb R^m$.
\end{assumption}

\begin{assumption}
\label{ass:graph}
The graph $\mathcal G$ is strongly connected.
\end{assumption}

\begin{remark}
Problem (\ref{form}) is well-known and has applications in distributed control, estimation, and machine learning  \cite{nedic2018distributed, yang2019survey}. Unlike most existing literature, we assume convexity only of $\sum_i f_i$ rather than of each $f_i$. This is motivated by the fact that perturbations used to address privacy concerns may render local costs nonconvex in distributed optimization \cite{nozari2016differentially, pang2022gradient}.
We should also emphasize that Assumption \ref{ass:graph} allows the communication graph to be weight-unbalanced as those in \cite{cheng2022distributed, li2017distributed, zhang2022fully}.
\end{remark}

Under Assumption \ref{ass:graph}, problem (\ref{form}) can be cast into 
\begin{equation}
\label{reform}
\min~ f(\bm x):= \sum\nolimits_{i\in \mathcal V} f_i(x_i), ~~
{\rm s.t.}~\bm L \bm x = \mathbf{0}
\end{equation}
where $x_i \in \mathbb R^m$ is the decision variable of agent $i$, $\tilde m = Nm$, $\bm x = [x_1^\top, \dots, x_N^\top]^\top \in \mathbb R^{\tilde m}$, $\bm{L} = L_G \otimes I_m$ and $L_G$ is the Laplacian matrix of the graph $\mathcal G$.

%%%%%
%%%%%-----------------
%\subsection{Distributed algorithm}

We define an augmented Lagrangian function of (\ref{reform}) as
\begin{equation}
\label{lag}
\mathcal L(\bm x, \bm z) = f(\bm x) + {\bm z}^\top \bm L {\bm x} + \frac 1 2 {\bm x}^\top {\bm L} {\bm x}.
\end{equation}
where $\bm z = [z_1^\top, \dots, z_N^\top]^\top \in \mathbb R^{\tilde m}$ is the dual variable.

Inspired by the Arrow–Hurwicz–Uzawa flow in \cite{arrow1958studies}, a distributed primal-dual algorithm is designed as
\begin{equation}
\label{dyn:pdc}
\left\{
\begin{aligned}
\dot {\bm x} &= - \gamma \nabla \bar f(\bm x) - \bm L \bm z - \alpha \bm L \bm x \\
\dot {\bm z}  &= \beta \bm L \bm x
\end{aligned}
\right.
\end{equation}
where the positive scalars $\alpha$, $\beta$ and $\gamma$ are the algorithm gains,
$\bar f_i(x_i) = \frac {1}{r_i} f_i(x_i)$,
$\nabla \bar f(\bm x) = [\nabla \bar f_1(x_1)^\top, \dots, \\
\nabla \bar f_N(x_N)^\top]^\top \in \mathbb R^{\tilde m}$, and $r = [r_i] \in \mathbb R^N$ is the left eigenvector associated with the zero eigenvalue of $L_G$ such that $\bm 1_N^\top r = 1$.

%\begin{equation}
%\label{dyn:pd}
%\left\{
%\begin{aligned}
%\dot x_i &= - \gamma \nabla \bar f_i(x_i) - \sum_{j \in \mathcal V} a_{ij}(\mu_i - \mu_j) - \alpha \sum_{j \in \mathcal V} a_{ij}(x_i - x_j) \\
%\dot \mu_i &= \sum_{j \in \mathcal V} a_{ij}(x_i - x_j).
%\end{aligned}
%\right.
%\end{equation}

Intuitively, for dynamics (\ref{dyn:pdc}), the primal variable $\bm x$ descends along the negative gradient of $\mathcal L$, while the dual variable $\bm z$ ascends along its gradient.
We note that (\ref{dyn:pdc}) is a fully distributed algorithm, in which each agent only has access to its local cost and communicates with its neighbors.
The modified local cost functions $\bar f_i$ are introduced to handle weight-unbalanced networks.
We suppose that the left eigenvalue $r$ is known in advance, which is similar to the algorithms in \cite{li2017distributed,cheng2022distributed}. In case of an unknown $r$, it can be estimated by employing the auxiliary dynamics in \cite{li2017cooperative, zhang2022fully}.
The following lemma clarifies the relationship between the optimal solution of (\ref{form}) and the equilibrium point of (\ref{dyn:pdc}).
\begin{lemma}
\label{lem:Eq:KKT}
Let Assumptions \ref{ass:convex} and \ref{ass:graph} hold. Then $x^*$ is an optimal solution to (\ref{form}) if and only if there exists $\bm z^* \in \mathbb R^{\tilde m}$ such that $(\bm x^*, \bm z^*)$ is an equilibrium of (\ref{dyn:pdc}) with $\bm x^* = \mathbf{1}_N \cdot x^* \in \mathbb R^{\tilde m}$, i.e.,
\begin{equation}
\label{Eq}
\begin{aligned}
\gamma \nabla \bar f(\bm x^*) + \bm L \bm z^* = \mathbf{0}, ~{\rm and}~ \bm L \bm x^* = \mathbf{0}.
\end{aligned}
\end{equation}
\end{lemma}
\textbf{Proof.} \emph{Sufficiency}:
Let $(\bm x^*, \bm z^*)$ be an equilibrium of (\ref{dyn:pdc}). Then (\ref{Eq}) holds.
By Assumption \ref{ass:graph}, there is $x^* \in \mathbb R^m$ such that $\bm x^* = \bm 1_N \otimes x^*$. Besides, $\gamma (r^\top \otimes I_m)\nabla \bar f(\bm x^*) + (r^\top \otimes I_m)\bm L \bm z^* = \gamma(\bm 1_N^\top \otimes I_m) \nabla f(\bm x^*) = \bm 0$, i.e.,
$\sum_{i \in \mathcal V} \nabla f_i(x^*) = \bm 0$. Thus, $x^*$ satisfies the first-order optimality condition of~(\ref{form}) and is the optimal solution.

\emph{Necessity}:
Let $x^*$ be a solution to (\ref{form}), and define $\bm x^* = \bm 1_N \otimes x^*$. 
Then $\sum_{i \in \mathcal V} \nabla f_i(x_i^*) = \bm 0$.
Under Assumption \ref{ass:graph}, the condition holds only if there exists $\bm z^*$ such that $\nabla \bar f(\bm x^*) + \bm L \bm z^* = \mathbf{0}$.
Hence, $(\bm x^*, \bm z^*)$ is an equilibrium of (\ref{dyn:pdc}), and this completes the proof.
$\hfill\square$

\begin{remark}
In fact, (\ref{dyn:pdc}) is a typical distributed continuous-time algorithm for consensus-based optimization, and its convergence has been established under convex local cost functions in  \cite{wang2010control,gharesifard2013distributed, kia2015distributed, zeng2016distributed, liang2019exponential}.
Extensions to weight-unbalanced digraphs were explored in \cite{zhang2022fully}, and further generalizations to uniformly jointly strongly connected time-varying topologies were studied in \cite{li2020input}.
Resorting to adaptive protocols, novel distributed algorithms were designed for convex sum of nonconvex functions in \cite{li2017distributed,li2020projection}.
Nevertheless, to the best of our knowledge, few works have been devoted to studying the roles of cost functions, algorithm gains and communication networks in the convergence of (\ref{dyn:pdc}). 
In this paper, we develop a systematic framework to analyze its performance, and investigate whether this algorithm can still be convergent under nonconvex local cost functions.
\end{remark}

%%%%%%
%%%%%%-----------------
\section{Main results}

In this section, we analyze the performance of dynamics (\ref{dyn:pdc}) from a dissipativity-based perspective.

To certify the performance of a complex nonlinear system, it is  effective to decompose it into multiple interconnected subsystems.
These simpler subsystems, together with their interconnections, reveal the structural properties of the system, and provide deeper insights into the mechanisms governing the overall system \cite{sepulchre2012constructive,khalil2002nonlinear,arcak2022compositional}.
Inspired by this idea, in order to investigate the convergence of (\ref{dyn:pdc}), we reformulate it as two interconnected subsystems. One subsystem replies on the Laplacian matrix $L_G$, while the other arises from the local costs.

Let $(\bm x^*, \bm z^*)$ be the equilibrium of (\ref{dyn:pdc}).
Define the error variables as $\tilde{\bm x} := \bm x - \bm x^*$, and $\tilde {\bm z} := \bm z - \bm z^*$. Recalling (\ref{dyn:pdc}) and (\ref{Eq}), we derive the corresponding error dynamics, and further decompose it into two feedback interconnected subsystems $\Sigma$ and $\Delta$ as
\begin{equation}
\label{dec:sig}
\begin{aligned}
\Sigma: 
\begin{cases}
\begin{bmatrix}
\dot {\tilde{\bm x}} \\
\dot {\tilde{\bm z}}
\end{bmatrix} \!\!\!\!\!\!&= 
\begin{bmatrix}
- \alpha \bm L,  &- \bm L \\
\beta \bm L, &\bm{0}
\end{bmatrix}
\begin{bmatrix}
\tilde{\bm x} \\
\tilde{\bm z}
\end{bmatrix}
+ \begin{bmatrix}
\gamma u \\
\bm 0
\end{bmatrix} \\
~~\tilde{\bm y} \!\!\!\!\!\!&= \tilde{\bm x}
\end{cases}
\end{aligned}
\end{equation}
where $u = - \Delta(\tilde {\bm y})$, $\bm y^* = \bm x^*$ and $\Delta(\tilde {\bm y}) = \nabla\bar f(\tilde{\bm y} + \bm y^*) - \nabla \bar f(\bm y^*)$. Furthermore, we obtain the block diagram representation of the error subsystems as shown in Fig. \ref{fig:decomp}.

\begin{figure}[htbp]
\centering
\includegraphics[width=0.6\linewidth]{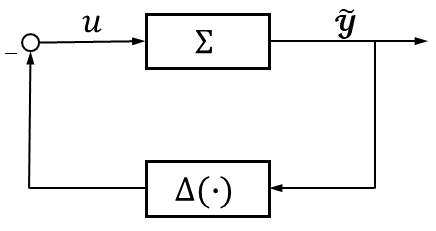}
\caption{Error dynamics as feedback interconnection of two subsystems.}
\label{fig:decomp}
\end{figure}

\begin{remark}
Fig. \ref{fig:decomp} implies that the error dynamics can be cast as a Lur'e system, consisting of a linear subsystem in feedback with a static nonlinearity. This differs from those in \cite{hatanaka2018passivity, li2020input, wang2010control}, where the dynamics is reformulated as a nonlinear subsystem with Laplacian matrices in static feedback.
This decomposition is advantageous in two aspects.
On the one hand, the linear subsystem can be 
analyzed using standard tools, while the nonlinear feedback can be characterized by properties of the costs.
On the other hand, this representation provides a clear system-level interpretation, explicitly revealing how the communication network and the cost functions jointly influence the performance of (\ref{dyn:pdc}).
\end{remark}

Throughout this paper, we define $\Pi = (\bm 1_N \otimes I_m)\cdot (r^\top \otimes I_m)$, and $\Pi_\perp = I_{\tilde m} - \Pi$. Accordingly, we decompose $\tilde {\bm x}$ as $\bar {\bm x} := \Pi \tilde {\bm x}$ and $\tilde{\bm x}_\perp := \Pi_\perp \tilde {\bm x}$.
It is worth noting that $\Pi$ induces a weighted average of $\tilde{\bm x}$, and that $\bar{\bm x}$ and $\tilde{\bm x}_\perp$ are generally not orthogonal under the standard Euclidean inner product. Such a decomposition will facilitate the analysis of dynamics (\ref{dyn:pdc}) under weight-unbalanced networks. In a similar manner, we define $\bar {\bm z}$, $\bar {\bm z}_\perp$, $\bar {\bm y}$ and $\bar {\bm y}_\perp$. 

We define the positive definite matrix $P \in \mathbb S_{++}^{2\tilde m}$ and introduce a matrix $M \in \mathbb R^{2\times 2}$ for later analysis as
\begin{equation}
\begin{aligned}
\label{def:PM}
P = \begin{bmatrix}
p_1, &p_2 \\
p_2, &1
\end{bmatrix} \otimes \bm R, ~{\rm and}~
M = \begin{bmatrix}
\alpha p_1-\beta p_2, &p_1 \\ 
p_1, &p_2
\end{bmatrix}
\end{aligned}
\end{equation}
where $\bm R := {\rm diag}\{r\} \otimes I_m$, and the scalars $p_1, p_2$ satisfy $p_1 > p_2^2$ and $p_2 > 0$.
Define $\underline  r := \min\{r_i\}$ and $\overline l := \max\{l_i\}$. Let $\underline \lambda(M)$ be the minimal eigenvalue of $M$, and $\rho(L_G)$ be the generalized algebraic connectivity of $\mathcal G$.

Denote $\tilde {\bm \eta} := [\tilde{\bm x}^\top, \tilde{\bm z}^\top]^\top$. 
We construct a storage function $V$ for the subsystem $\Sigma$ as
\begin{equation}
\begin{aligned}
V(\tilde {\bm x}, \tilde {\bm z}) = \frac 1 2
{\tilde{\bm \eta}}^\top
P {\tilde{\bm \eta}}.
\end{aligned}
\end{equation}
The next theorem characterizes the dissipativity properties of the subsystems $\Sigma$ and $\Delta$. 

\begin{theorem}
\label{thm:passivity}
Under Assumption \ref{ass:convex} and \ref{ass:graph}, the following statements are satisfied.
\begin{enumerate}[i)]
\item 
If $\alpha p_2 - \beta = p_1$, it holds for the subsystems $\Sigma$ that
\begin{equation}
\begin{aligned}
\label{pf:th1:1}
\dot V \le &
- \underline\lambda(M) \rho(L_G) \underline r \Vert \tilde{\bm y}_\perp \Vert^2 - \Big(\underline \lambda(M) \rho(L_G) \underline r - \frac {p_2 \gamma}{2\delta}\Big) \\
&\cdot \Vert \tilde{\bm z}_\perp \Vert^2
+ p_1 \gamma \langle \tilde{\bm y}, \bm R u\rangle
+ \frac {\delta p_2 \gamma}{2}\Vert \bm R u \Vert^2.
\end{aligned}
\end{equation}
where $\delta$ is an arbitrary positive scalar.

\item For the subsystems $\Delta$, we have
\begin{equation}
\begin{aligned}
\label{pf:th1:2}
\langle \tilde {\bm y}, \bm R\Delta(\tilde{\bm y}) \rangle 
\ge \frac {\mu}{2}\Vert \bar{\bm y} \Vert^2 - \big(\overline l + \frac {2 \overline l^2}{\mu}\big)\Vert \tilde{\bm y}_{\perp}\Vert^2.
\end{aligned}
\end{equation}
\end{enumerate}
\end{theorem}

\textbf{Proof.}
Recalling (\ref{dec:sig}) gives
\begin{equation*}
\begin{aligned}
\dot V = 
- \frac 12{\tilde{\bm \eta}}^\top
Q {\tilde{\bm \eta}}
+ p_1 \gamma \langle \tilde{\bm x}, \bm R u\rangle
+ p_2 \gamma \langle \tilde{\bm z}, \bm R u\rangle
\end{aligned}
\end{equation*}
where
\begin{equation*}
\begin{aligned}
Q \!=\! \begin{bmatrix}
(\alpha p_1 \!-\! \beta p_2)(\bm L^\top\bm R \!+\! \bm R \bm L),
& (\alpha p_2 \!- \! \beta) \bm L^\top \bm R \!+\! p_1 \bm R \bm L \\
(\alpha p_2 \!-\!\beta) \bm R \bm L \!+\! p_1  \bm L^\top \bm R
&p_2(\bm L^\top\bm R \!+\! \bm R \bm L)
\end{bmatrix}
\end{aligned}.
\end{equation*}
If $\alpha p_2 - \beta = p_1$, then 
$Q = M \otimes(\bm L^\top\bm R+\bm R\bm L)$, where $M$ is defined in (\ref{def:PM}).
Moreover, it is straightforward to verify that $M \succ \bm 0$ if $P \succ \bm 0$. Consequently, $Q \succeq \bm 0$ since $\bm L^\top\bm R + \bm R \bm L \succeq \bm{0}$.
Moreover,  there exists $\tilde Q \in \mathbb R^{\tilde m \times \tilde m}$ such that $\bm L^\top\bm R+\bm R\bm L = \tilde Q^\top \tilde Q$, and 
$Q = (I_2 \otimes \tilde Q^\top) (M \otimes I_{\tilde m})(I_2 \otimes \tilde Q)$.
Then ${\tilde{\bm \eta}}^\top
Q {\tilde{\bm \eta}} \ge \underline \lambda(M) \tilde{\bm \eta}^\top [I_2 \otimes (\bm L^\top\bm R + \bm R \bm L)] \tilde{\bm \eta}$.
Furthermore, in light of Lemma \ref{lem:graph}~iii), we have
$${\tilde{\bm \eta}}^\top Q {\tilde{\bm \eta}} \ge 2\underline\lambda(M) \rho(L_G)(\tilde{\bm x}_\perp^\top \bm R \bm \tilde{\bm x}_\perp  + \tilde{\bm z}_\perp^\top \bm R \tilde{\bm z}_\perp).$$ 
As a result, we derive
\begin{equation*}
\begin{aligned}
\dot V \le &
- \underline\lambda(M) \rho(L_G) \underline r \big(\Vert \tilde{\bm x}_\perp \Vert^2 + \Vert \tilde{\bm z}_\perp \Vert^2 \big) \\
&+ p_1 \gamma \langle \tilde{\bm x}, \bm R u\rangle
+ p_2 \gamma \langle \tilde{\bm z}, \bm R u\rangle.
\end{aligned}
\end{equation*}
By (\ref{dyn:pdc}), it is clear that $(r^\top \otimes I_{\tilde m}) \dot {\bm z} = \bm{0}$, and then,
$\Vert \tilde {\bm z} \Vert^2 = \Vert \tilde{\bm z}_\perp \Vert^2$.
For any $\delta > 0$, we have
$\langle \tilde{\bm z}, \bm R u\rangle \le 
\frac {1}{2 \delta} \Vert \bm z_\perp\Vert^2 + \frac {\delta}{2}\Vert \bm R u \Vert^2.$
Therefore,
\begin{equation*}
\begin{aligned}
\dot V \le &
- \underline\lambda(M) \rho(L_G) \underline r  \Vert \tilde{\bm x}_\perp \Vert^2 - \Big(\underline\lambda(M) \rho(L_G) \underline r - \frac {p_2 \gamma}{2\delta}\Big)\\
&\cdot \Vert \tilde{\bm z}_\perp \Vert^2
+ p_1 \gamma \langle \tilde{\bm x}, \bm R u\rangle
+ \frac {\delta p_2 \gamma}{2}\Vert \bm R u \Vert^2.
\end{aligned}
\end{equation*}
Part i) follows immediately from $\tilde {\bm y} = \tilde {\bm x}$.

By the definition of $\bar f_i$, we obtain
\begin{equation*}
\begin{aligned}
\langle \tilde {\bm y}, \bm R\Delta(\tilde{\bm y}) \rangle &= \big\langle \tilde {\bm y}, \nabla f(\tilde{\bm y} + \bm y^*) - \nabla f(\bm y^*)\big\rangle \\
& = \big\langle \bar {\bm y} + \tilde{\bm y}_\perp, \nabla f(\tilde{\bm y} + \bm y^*) - \nabla f(\bar {\bm y} + \bm y^*) \\
&~~~~~~~~~~~~~~ + \nabla f(\bar {\bm y} + \bm y^*) - \nabla f(\bm y^*) \big\rangle \\
&\ge - \overline l \Vert  \tilde{\bm y}_{\perp}\Vert^2 - 2 \overline l \Vert \tilde{\bm y}_{\perp} \Vert \cdot \Vert \bar{\bm y}\Vert + \mu \Vert \bar{\bm y} \Vert^2 \\
&\ge \frac {\mu}{2}\Vert \bar{\bm y} \Vert^2 - \Big(\overline l + \frac {2\overline l^2}{\mu}\Big)\Vert \tilde{\bm y}_{\perp}\Vert^2.
\end{aligned}
\end{equation*}
where the first inequality follows from the Lipschitz continuity of $\nabla f_i$ and the strong convexity of $f$, and the second inequality is derived using Young’s inequality. Thus, part ii) holds, and this completes the proof.
$\hfill\square$

\begin{remark}
Theorem \ref{thm:passivity}~i) indicates that by appropriately selecting the gains $\alpha$, $\beta$ and $\gamma$, the subsystem $\Sigma$ is dissipative with respect to a suitable supply function. 
In a special case $\gamma = 0$, the state $\tilde {\bm \eta}$ converges under proper $\alpha$ and $\beta$, and the result generalizes \cite[Lemma 5.3]{gharesifard2013distributed} to weight-unbalanced graphs.
It is evident that the operator $\Delta$ is not passive due to the nonconvexity of $f_i$. However, Theorem \ref{thm:passivity}~ii) fully characterizes its properties under Assumption \ref{ass:convex}.
Theorem \ref{thm:passivity} provides new insights into the roles of the network topology, algorithms gains and the cost functions in the convergence of (\ref{dyn:pdc}).
To the best of our knowledge, this is the first work to investigate the dispassivity properties of $\Sigma$ and $\Delta$ from this perspective.
\end{remark}

In what follows, we establish the convergence of (\ref{dyn:pdc}).
\begin{theorem}
\label{thm:convergence}
Consider dynamics (\ref{dyn:pdc}).
Let Assumptions \ref{ass:convex} and \ref{ass:graph} hold.
Suppose that $\alpha^2 > 4 \beta$. 
Then there exists a gain $\gamma>0$ or a communication graph $\mathcal G$
such that the trajectory  $\big(\bm x(t), \bm z(t)\big)$ converges to an equilibrium $(\bm x^*, \bm z^*)$ of (\ref{dyn:pdc}) with an exponential rate, where $\bm x^* = \bm 1_N \otimes x^*$, and $x^*$ is the optimal solution to (\ref{form}).
\end{theorem}
\textbf{Proof.}
Since $\Vert \tilde {\bm y}\Vert^2 \le 2\big(\Vert \bar{\bm y}\Vert^2 + \Vert\tilde{\bm y}_\perp\Vert^2\big)$, it follows from Assumption \ref{ass:convex} that
\begin{equation}
\begin{aligned}
\label{pf:th1:3}
\Vert \bm R \Delta(\tilde{\bm y})\Vert^2 &= \Vert \bm \nabla f(\tilde{\bm y} + \bm y^*) - \nabla f(\bm y^*)\Vert^2 \\
&\le 2l^2 \big(\Vert \bar{\bm y} \Vert^2 + \Vert \tilde{\bm y}_\perp \Vert^2 \big).
\end{aligned}
\end{equation}
Note that $u = - \Delta(\tilde {\bm y})$. Take $\delta = \frac{p_1 \mu}{4 p_2 l^2}$.
Plugging (\ref{pf:th1:2}), (\ref{pf:th1:3}) into (\ref{pf:th1:1}), we obtain
\begin{equation*}
\begin{aligned}
\dot V \le & - \frac{p_1 \gamma \mu}{4} \Vert \bar {\bm y} \Vert^2 - \Big(\underline\lambda(M) \rho(L_G) \underline r - \frac {2 p_2^2 \gamma \overline l^2}{p_1 \mu}\Big)\Vert \tilde{\bm z}_\perp \Vert^2 \\
&-\! \Big(\underline\lambda(M) \rho(L_G) \underline r - \frac{p_1 \gamma \mu}{4} - p_1 \gamma \big(\overline l + \frac {2 \overline l^2}{\mu}\big)\Big)\Vert \tilde{\bm y}_\perp \Vert^2.
\end{aligned}
\end{equation*}
Suppose that $L_G$ is given and $\alpha^2 > 4 \beta$. 
Set $p_2 = \frac \alpha 2$, and $p_1 = \alpha p_2 - \beta = \frac {\alpha^2}{2} - \beta$.
Then it is clear that both $P$ and $M$ are positive definite.
Note that $\Vert \tilde {\bm x}\Vert = \Vert \tilde{\bm y}\Vert$ and $\Vert \tilde {\bm z}\Vert = \Vert \tilde{\bm z}_\perp\Vert$.
If
$$\gamma < \min\left\{\frac{\underline\lambda(M) \rho(L_G) \underline r p_1 \mu}{2p_2^2 \overline l^2},  \frac{4\underline\lambda(M) \rho(L_G) \underline r \mu}{p_1 \mu^2 + 4p_1 (\overline l \mu + 2 \overline l^2)}\right\},
$$
then there exists $\tilde \delta > 0$ such that
\begin{equation}
\begin{aligned}
\dot V \le - \tilde \delta \big(\Vert \tilde{\bm x}\Vert^2 + \Vert \tilde{\bm z}\Vert^2\big) \le -\frac{2\tilde \delta}{\bar{\lambda}(P)} V
\end{aligned}
\end{equation}
where $\bar{\lambda}(P)$ denotes the maximal eigenvalue of $P$.  Therefore, the trajectory  $\big(\bm x(t), \bm z(t)\big)$ converges to the equilibrium $(\bm x^*, \bm z^*)$ of (\ref{dyn:pdc}) with an exponential rate. 

Given $\alpha$, $\beta$ and $\gamma$ such that $\alpha^2 > 4 \beta$, we also set $p_1 = \frac {\alpha^2}{2} - \beta$, and $p_2 = \frac \alpha 2$. By properly designing the communication network such that 
$$\rho(L_G) \underline r > \left\{\frac {2 p_2^2 \gamma l^2}{p_1 \mu \underline\lambda(M)}, \frac{p_1 \gamma}{\underline\lambda(M)} \left(\frac{\mu}{4} + \frac {\overline l \mu + 2 \overline l^2}{\mu}\right)\right\},$$ for example, by uniformly scaling the adjacency matrix of $\mathcal G$, it can hold that $\dot V \le - \hat \delta V$ for some $\hat \delta > 0$. Hence, the trajectory $\big(\bm x(t), \bm z(t)\big)$  converges exponentially to the equilibrium $(\bm x^*, \bm z^*)$ of (\ref{dyn:pdc}).

Resorting to Lemma 1, we derive $\bm x^* = 1_N^\top \otimes x^*$, and thus, the proof is completed.
$\hfill\square$

\begin{remark}
Theorem \ref{thm:convergence} demonstrates that dynamics (\ref{dyn:pdc}) achieves the optimal solution of (\ref{form}) under appropriate algorithm gains or communication graphs even though the local cost functions are nonconvex. This generalizes existing results reported in \cite{wang2010control, gharesifard2013distributed, kia2015distributed, zeng2016distributed, liang2019exponential}, and significantly extends the applicability of primal-dual methods.
The key reason underlying this extension is that the overall system is decomposed into two coupled subsystems, whose intrinsic properties are fully explored. In particular, although $\Delta$ is not passive, $\Sigma$ can be rendered passive excess by properly selecting the gains or by designing the graph. The passivity shortage of $\Delta$ can be compensated by $\Sigma$, which guarantees the convergence of the overall system.
\end{remark}

%%%%%
%%%%%-------------------
\section{Numerical simulations}

In this section, we carry out simulations for illustration.

Consider a network consisting of five agents. Two weight-unbalanced and strongly connected communication networks are employed, as shown in Fig. \ref{fig:grah}. In Fig. \ref{fig:grah}(a), the adjacency matrix is set to $a_{ij} = 1$ if $a_{ij} \not = 0$, whereas in Fig. \ref{fig:grah}(b), all nonzero entries are set to $4$.
The local cost functions in (\ref{form}) are given by
$f_1(x) = x_{(1)}^2 + x_{(1)} x_{(2)} + 5  x_{(3)}^2 - x_{(4)}^2 + e^{x_{(1)}}$,
$f_2(x) = 2x_{(2)}^2 - 2 x_{(1)} x_{(2)} - x_{(3)}^2 - e^{x_{(1)}}$,
$f_3(x) = x_{(1)}^2 - 2 x_{(3)}^2 + x_{(2)} x_{(3)} - \sin(x_{(4)})$,
$f_4(x) = -x_{(1)}^2 + 3  x_{(4)}^2 + \sin(x_{(4)})$,
$f_5(x) = - x_{(2)}^2 +  x_{(4)}^2$,
where $x_{(k)}$ denotes the $k$th component of $x$.
It is clear that none of $f_i$ is convex, but the global cost $\tilde f(x)$ is strongly convex.

For dynamics (\ref{dyn:pdc}), we first consider the case $\gamma = 0$, which corresponds to the absence of local costs. Fig. \ref{fig:consensus} compares the trajectories of $\bm x^\top \bm L \bm x$ and $\bm z^\top \bm L \bm z$ under different choices of $\alpha$ and $\beta$.
Given $\beta = 1$, the dynamics fails to converge when $\alpha = 1$, but converges when $\alpha = 5$.
The results imply that the condition $\alpha^2 > 4\beta$ is critical to guarantee the convergence of (\ref{dyn:pdc}).

We then verify that for given $\alpha$ and $\beta$, the gain $\gamma$ plays an important role on the convergence of algorithm (\ref{dyn:pdc}).
Fig. \ref{fig:converge:gamma} shows the trajectories of $f(\bm x)$ under different choices of $\gamma$.
The results indicate that (\ref{dyn:pdc}) is generally not convergent under Assumption \ref{ass:convex}. However, the convergence can be achieved by properly selecting $\gamma$.

Finally, we show that the communication topology also influences the convergence of (\ref{dyn:pdc}). Using the same gains as those in Fig. \ref{fig:converge:gamma}(b), we conduct experiments with the graph in Fig. \ref{fig:grah}(b). Fig. \ref{fig:converge:graph}(a) demonstrates that (\ref{dyn:pdc}) is convergent under this graph.
Furthermore, Fig. \ref{fig:converge:graph}(b) plots the trajectory of $\log(\Vert \bm x - \bm x^*\Vert)$, and indicates the exponential convergence of \eqref{dyn:pdc}.

\begin{figure}[htbp]
\centering
\includegraphics[width=0.5\linewidth]{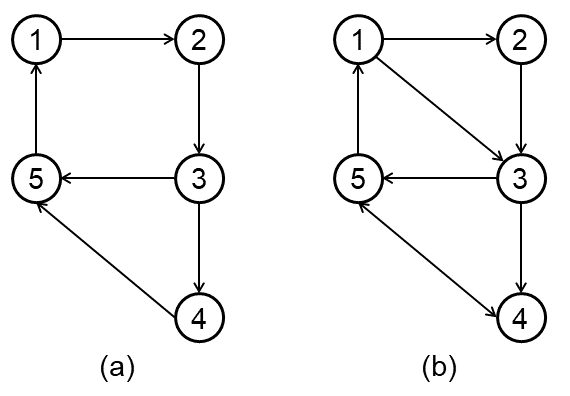}
\caption{The communication graphs of the five-agent network.}
\label{fig:grah}
\end{figure}

\begin{figure}[htbp]
\centering
\includegraphics[width=1\linewidth]{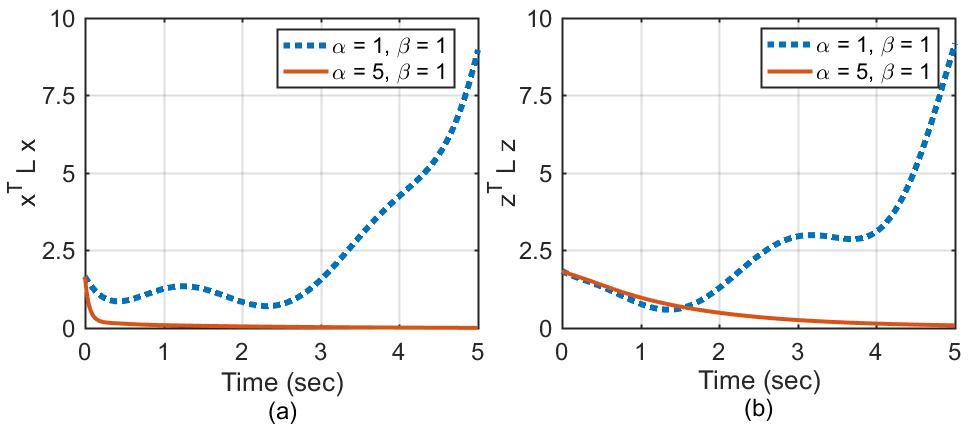}
\caption{Comparison of trajectories under different choices of $\alpha$ and $\beta$: (a) $\bm x^\top \bm L \bm x$; (b) $\bm z^\top \bm L \bm z$.}
\label{fig:consensus}
\end{figure}

\begin{figure}[htbp]
\centering
\includegraphics[width=1\linewidth]{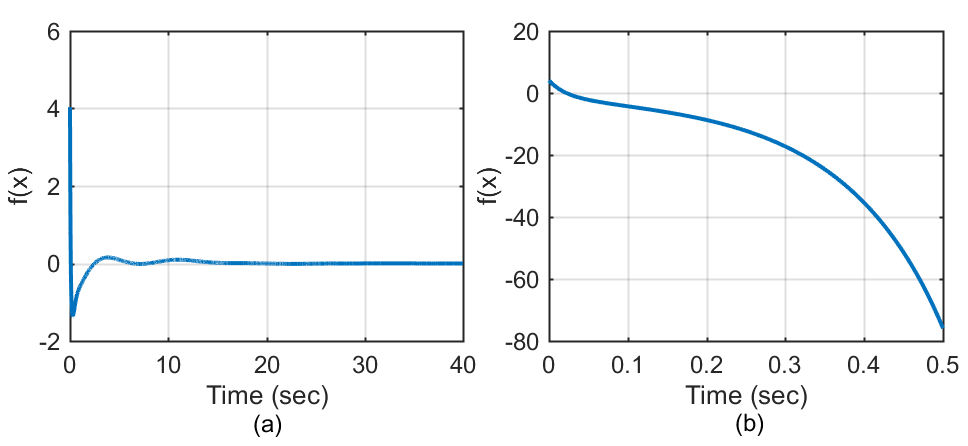}
\caption{Given $\alpha = 5$ and $\beta = 1$, trajectories of $f(\bm x)$ under different $\gamma$: (a) $\gamma = 0.1$; (b) $\gamma = 0.5$.}
\label{fig:converge:gamma}
\end{figure}

\begin{figure}[htbp]
\centering
\includegraphics[width=1\linewidth]{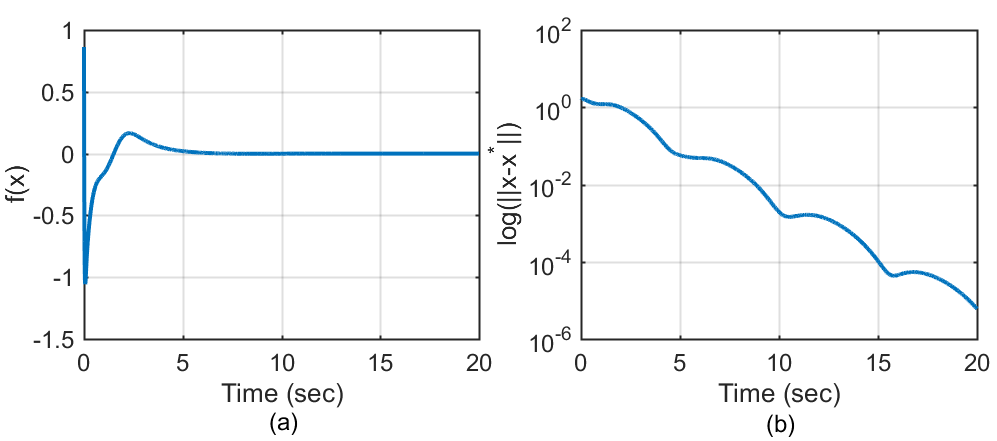}
\caption{Trajectory with $\alpha = 5$, $\beta = 1$ and $\gamma = 0.5$ under the network in Fig. \ref{fig:grah}(b): (a) $f(\bm x)$; (b) $\log(\Vert \bm x - \bm x^*\Vert)$.}
\label{fig:converge:graph}
\end{figure}

%%%%
%%%%-----------------
\section{Conclusion}

This paper investigated the convergence of a primal-dual algorithm to solve a distributed optimization problem over weight-unbalanced networks, where the global cost function is convex, while the individual local cost functions may be nonconvex.
The algorithm was decomposed as a Lur'e system.
By fully exploring the properties of the Lapalacian matrix of the graph as well as the local costs, we established that the linear subsystem is dissipative, whereas the static gradient feedback is not passive. It was found that the algorithm converges to an equilibrium with an exponential rate under proper algorithm gains or communication networks. Finally, numerical simulations were provided for illustration.

\bibliographystyle{IEEEtran}
\bibliography{references.bib}

@article{yuan2024multi,
	title={Multi-agent online optimization},
	author={Yuan, Deming and Proutiere, Alexandre and Shi, Guodong},
	journal={Foundations and Trends in Optimization},
	volume={7},
	number={2-3},
	pages={81--263},
	year={2024},
	publisher={Emerald Publishing Limited}
}

@article{neal2011distributed,
	title={Distributed optimization and statistical learning via the alternating direction method of multipliers},
	author={Neal, Parikh and Eric, Chu and Borja, Peleato and Jonathan, Eckstein},
	journal={Foundations and Trends{\textregistered} in Machine learning},
	volume={3},
	number={1},
	pages={1--122},
	year={2011},
	publisher={Emerald Publishing Limited}
}

@book{khalil2002nonlinear,
	title={Nonlinear systems},
	author={Khalil, Hassan K},
	volume={3},
	year={2002},
	publisher={Prentice hall Upper Saddle River, NJ}
}

@article{sahu2020decentralized,
	title={Decentralized zeroth-order constrained stochastic optimization algorithms: Frank--Wolfe and variants with applications to black-box adversarial attacks},
	author={Sahu, Anit Kumar and Kar, Soummya},
	journal={Proceedings of the IEEE},
	volume={108},
	number={11},
	pages={1890--1905},
	year={2020},
	publisher={IEEE}
}

@article{duchi2011dual,
	title={Dual averaging for distributed optimization: Convergence analysis and network scaling},
	author={Duchi, John C and Agarwal, Alekh and Wainwright, Martin J},
	journal={IEEE Transactions on Automatic control},
	volume={57},
	number={3},
	pages={592--606},
	year={2011},
	publisher={IEEE}
}

@article{lei2016primal,
	title={Primal--dual algorithm for distributed constrained optimization},
	author={Lei, Jinlong and Chen, Han-Fu and Fang, Hai-Tao},
	journal={Systems \& Control Letters},
	volume={96},
	pages={110--117},
	year={2016},
	publisher={Elsevier}
}

@article{yi2020distributed,
	title={Distributed online convex optimization with time-varying coupled inequality constraints},
	author={Yi, Xinlei and Li, Xiuxian and Xie, Lihua and Johansson, Karl H},
	journal={IEEE Transactions on Signal Processing},
	volume={68},
	pages={731--746},
	year={2020},
	publisher={IEEE}
}

@article{liang2019exponential,
	title={Exponential convergence of distributed primal--dual convex optimization algorithm without strong convexity},
	author={Liang, Shu and Yin, George and others},
	journal={Automatica},
	volume={105},
	pages={298--306},
	year={2019},
	publisher={Elsevier}
}

@book{li2017cooperative,
	title={Cooperative control of multi-agent systems: {A} consensus region approach},
	author={Li, Zhongkui and Duan, Zhisheng},
	year={2017},
	publisher={CRC press}
}

@book{sepulchre2012constructive,
	title={Constructive nonlinear control},
	author={Sepulchre, Rodolphe and Jankovic, Mrdjan and Kokotovic, Petar V},
	year={2012},
	publisher={Springer Science \& Business Media}
}

@article{lessard2022analysis,
	title={The analysis of optimization algorithms: A dissipativity approach},
	author={Lessard, Laurent},
	journal={IEEE Control Systems Magazine},
	volume={42},
	number={3},
	pages={58--72},
	year={2022},
	publisher={IEEE}
}

@article{welikala2025decentralized,
	title={Decentralized co-design of distributed controllers and communication topologies for vehicular platoons: {A} dissipativity-based approach},
	author={Welikala, Shirantha and Song, Zihao and Lin, Hai and Antsaklis, Panos J},
	journal={Automatica},
	volume={174},
	pages={112118},
	year={2025},
	publisher={Elsevier}
}

@article{li2025passivity,
	title={Passivity-based Gradient-Play Dynamics for Distributed Generalized Nash Equilibrium Seeking},
	author={Li, Weijian and Pavel, Lacra},
	journal={IEEE Transactions on Automatic Control},
	year={2025},
	publisher={IEEE}
}

@article{yamashita2020passivity,
	title={Passivity-based generalization of primal--dual dynamics for non-strictly convex cost functions},
	author={Yamashita, Shunya and Hatanaka, Takeshi and Yamauchi, Junya and Fujita, Masayuki},
	journal={Automatica},
	volume={112},
	pages={108712},
	year={2020},
	publisher={Elsevier}
}

@article{zakeri2022passivity,
	title={Passivity measures in cyberphysical systems design: An overview of recent results and applications},
	author={Zakeri, Hasan and Antsaklis, Panos J},
	journal={IEEE Control Systems Magazine},
	volume={42},
	number={2},
	pages={118--130},
	year={2022},
	publisher={IEEE}
}

@article{wen2004unifying,
	title={A unifying passivity framework for network flow control},
	author={Wen, John T and Arcak, Murat},
	journal={IEEE Transactions on Automatic Control},
	volume={49},
	number={2},
	pages={162--174},
	year={2004},
	publisher={IEEE}
}

@article{arcak2022compositional,
	title={Compositional design and verification of large-scale systems using dissipativity theory: Determining stability and performance from subsystem properties and interconnection structures},
	author={Arcak, Murat},
	journal={IEEE Control Systems Magazine},
	volume={42},
	number={2},
	pages={51--62},
	year={2022},
	publisher={IEEE}
}

@book{arrow1958studies,
	title={Studies in linear and non-linear programming},
	author={Arrow, Kenneth Joseph and Hurwicz, Leonid and Uzawa, Hirofumi and Chenery, Hollis Burnley and Johnson, Selmer and Karlin, Samuel},
	volume={2},
	year={1958},
	publisher={Stanford University Press Stanford}
}

@article{nedic2009distributed,
	title={Distributed subgradient methods for multi-agent optimization},
	author={Nedic, Angelia and Ozdaglar, Asuman},
	journal={IEEE Transactions on Automatic Control},
	volume={54},
	number={1},
	pages={48--61},
	year={2009},
	publisher={IEEE}
}

@article{yang2019survey,
	title={A survey of distributed optimization},
	author={Yang, Tao and Yi, Xinlei and Wu, Junfeng and Yuan, Ye and Wu, Di and Meng, Ziyang and Hong, Yiguang and Wang, Hong and Lin, Zongli and Johansson, Karl H},
	journal={Annual Reviews in Control},
	volume={47},
	pages={278--305},
	year={2019},
	publisher={Elsevier}
}

@article{nedic2018distributed,
	title={Distributed optimization for control},
	author={Nedi{\'c}, Angelia and Liu, Ji},
	journal={Annual Review of Control, Robotics, and Autonomous Systems},
	volume={1},
	number={1},
	pages={77--103},
	year={2018},
	publisher={Annual Reviews}
}

@article{li2017distributed,
	title={Distributed adaptive convex optimization on directed graphs via continuous-time algorithms},
	author={Li, Zhenhong and Ding, Zhengtao and Sun, Junyong and Li, Zhongkui},
	journal={IEEE Transactions on Automatic Control},
	volume={63},
	number={5},
	pages={1434--1441},
	year={2017},
	publisher={IEEE}
}

@article{pang2022gradient,
	title={A gradient-free distributed optimization method for convex sum of nonconvex cost functions},
	author={Pang, Yipeng and Hu, Guoqiang},
	journal={International Journal of Robust and Nonlinear Control},
	volume={32},
	number={14},
	pages={8086--8101},
	year={2022},
	publisher={Wiley Online Library}
}

@article{nozari2016differentially,
	title={Differentially private distributed convex optimization via functional perturbation},
	author={Nozari, Erfan and Tallapragada, Pavankumar and Cort{\'e}s, Jorge},
	journal={IEEE Transactions on Control of Network Systems},
	volume={5},
	number={1},
	pages={395--408},
	year={2016},
	publisher={IEEE}
}

@article{li2020projection,
	title={Projection-free distributed optimization with nonconvex local objective functions and resource allocation constraint},
	author={Li, Dewen and Li, Ning and Lewis, Frank},
	journal={IEEE Transactions on Control of Network Systems},
	volume={8},
	number={1},
	pages={413--422},
	year={2020},
	publisher={IEEE}
}

@article{kia2015distributed,
	title={Distributed convex optimization via continuous-time coordination algorithms with discrete-time communication},
	author={Kia, Solmaz S and Cort{\'e}s, Jorge and Mart{\'\i}nez, Sonia},
	journal={Automatica},
	volume={55},
	pages={254--264},
	year={2015},
	publisher={Elsevier}
}

@article{zeng2016distributed,
	title={Distributed continuous-time algorithm for constrained convex optimizations via nonsmooth analysis approach},
	author={Zeng, Xianlin and Yi, Peng and Hong, Yiguang},
	journal={IEEE Transactions on Automatic Control},
	volume={62},
	number={10},
	pages={5227--5233},
	year={2016},
	publisher={IEEE}
}

@article{li2020input,
	title={Input-feedforward-passivity-based distributed optimization over jointly connected balanced digraphs},
	author={Li, Mengmou and Chesi, Graziano and Hong, Yiguang},
	journal={IEEE Transactions on Automatic Control},
	volume={66},
	number={9},
	pages={4117--4131},
	year={2020},
	publisher={IEEE}
}

@article{gharesifard2013distributed,
	title={Distributed continuous-time convex optimization on weight-balanced digraphs},
	author={Gharesifard, Bahman and Cort{\'e}s, Jorge},
	journal={IEEE Transactions on Automatic Control},
	volume={59},
	number={3},
	pages={781--786},
	year={2013},
	publisher={IEEE}
}

@article{hatanaka2018passivity,
	title={Passivity-based distributed optimization with communication delays using PI consensus algorithm},
	author={Hatanaka, Takeshi and Chopra, Nikhil and Ishizaki, Takayuki and Li, Na},
	journal={IEEE Transactions on Automatic Control},
	volume={63},
	number={12},
	pages={4421--4428},
	year={2018},
	publisher={IEEE}
}

@inproceedings{wang2010control,
	title={Control approach to distributed optimization},
	author={Wang, Jing and Elia, Nicola},
	booktitle={2010 48th Annual Allerton Conference on Communication, Control, and Computing (Allerton)},
	pages={557--561},
	year={2010},
	organization={IEEE}
}

@article{cheng2022distributed,
	title={Distributed gradient tracking for unbalanced optimization with different constraint sets},
	author={Cheng, Songsong and Liang, Shu and Fan, Yuan and Hong, Yiguang},
	journal={IEEE Transactions on Automatic Control},
	volume={68},
	number={6},
	pages={3633--3640},
	year={2022},
	publisher={IEEE}
}

@article{li2020cooperative,
	title={Cooperative source seeking via networked multi-vehicle systems},
	author={Li, Zhuo and You, Keyou and Song, Shiji},
	journal={Automatica},
	volume={115},
	pages={108853},
	year={2020},
	publisher={Elsevier}
}

@article{zhang2022fully,
	title={Fully distributed algorithm for resource allocation over unbalanced directed networks without global Lipschitz condition},
	author={Zhang, Jin and Liu, Lu and Wang, Xinghu and Ji, Haibo},
	journal={IEEE Transactions on Automatic Control},
	volume={68},
	number={8},
	pages={5119--5126},
	year={2022},
	publisher={IEEE}
}

@article{davydov2025time,
	title={Time-varying convex optimization: A contraction and equilibrium tracking approach},
	author={Davydov, Alexander and Centorrino, Veronica and Gokhale, Anand and Russo, Giovanni and Bullo, Francesco},
	journal={IEEE Transactions on Automatic Control},
	year={2025},
	publisher={IEEE}
}

@article{gokhale2023contractivity,
	title={Contractivity of distributed optimization and {Nash} seeking dynamics},
	author={Gokhale, Anand and Davydov, Alexander and Bullo, Francesco},
	journal={IEEE Control Systems Letters},
	volume={7},
	pages={3896--3901},
	year={2023},
	publisher={IEEE}
}

\end{document}